\documentclass[envcountsame,envcountresetchap]{sl}

\usepackage{slsec}     
\usepackage{stud_log} 
\usepackage{slfoot}
\usepackage{slthm}     

\usepackage{mathptmx}       
\usepackage{helvet}         
\usepackage{courier}        
\usepackage{type1cm}        
\usepackage[bottom]{footmisc} 

\usepackage{amssymb,amsmath}
\usepackage{array,booktabs}
\usepackage[shortlabels]{enumitem}
\usepackage{tikz}

\usetikzlibrary{automata, cd, positioning}
\usepackage{proof}
\usepackage{mathtools} 
\usepackage{color} 

\usepackage[normalem]{ulem} 

\DeclareMathAlphabet{\mathcal}{OMS}{cmsy}{m}{n}
\newlist{enummarg}{enumerate}{1}
\setlist[enummarg]{label*=\arabic*., align=left, leftmargin=0pt, labelwidth=*, itemsep=3pt}
\newlist{itemmarg}{itemize}{1}
\setlist[itemmarg]{align=left, leftmargin=0pt, labelwidth=*, itemsep=3pt}
\SetEnumitemKey{newpar}{before=\strut}

\usepackage{hyphenat}

\newtheorem{theorem}{Theorem}[section]  
\newtheorem{corollary}[theorem]{Corollary}
\newtheorem{lemma}[theorem]{Lemma}
\newtheorem{proposition}[theorem]{Proposition}

\theoremstyle{definition}

 \renewcommand{\qedsymbol}{\openbox}

\newcommand\pair[1]{\langle#1\rangle}
\newcommand{\m}{\mathbf} 
\newcommand{\pc}[1]{\ensuremath{{\sf #1}}}	
\newcommand\ld{\backslash}
\newcommand\rd{/}

\newcommand{\ut}{{\mathrm e}}
\newcommand{\zr}{{\mathrm f}}
\newcommand{\jn}{\vee}
\newcommand{\mt}{\wedge}
\newcommand{\pd}{\cdot}
\newcommand{\eq}{\approx}
\newcommand{\vty}[1]{\mathcal{#1}} 
\newcommand{\altnot}{\mathop{\sim}}
\newcommand{\daltnot}{\altnot\altnot}
\newcommand\et{\DOTSB\;\mathrel{\&}\;}

\newcommand*{\seq}{{\vphantom{A}\Rightarrow{\vphantom{A}}}}
\newcommand*{\rseq}{\Rightarrow}
\newcommand*{\pfa}[1]{\mbox{\scriptsize $#1$}} 	
\newcommand{\PI}{\mathrm{\Pi}}
\newcommand{\De}{\mathrm{\Delta}}
\newcommand{\Ga}{\mathrm{\Gamma}}
\newcommand{\idr}{(\textsc{id})}
\newcommand{\cutr}{\textup{\sc (cut)}}

\newcommand{\flr}{(\zr\!\rseq)}
\newcommand{\frr}{(\rseq\!\zr)}
\newcommand{\tlr}{(\ut\!\rseq)}
\newcommand{\trr}{(\rseq\!\ut)}
\newcommand{\olr}{(\jn\!\rseq)}
\newcommand{\orr}{(\rseq\!\jn)}
\newcommand{\alr}{(\mt\!\rseq)}
\newcommand{\arr}{(\rseq\!\mt)}
\newcommand{\pdlr}{(\pd\!\rseq)}
\newcommand{\pdrr}{(\rseq\!\pd)}
\newcommand{\ilr}{(\to\rseq)}
\newcommand{\irr}{(\rseq\to)}
\newcommand{\ldlr}{({\ld}{\rseq})}
\newcommand{\ldrr}{(\rseq\!\!\ld)}
\newcommand{\rdlr}{(\rd\!\rseq)}
\newcommand{\rdrr}{(\rseq\!\!\rd)}
\newcommand{\elr}{(\textsc{el})}
\newcommand{\err}{(\textsc{er})}
\newcommand{\lgpw}{(\LG\textsc{-w})}
\newcommand{\ablgpw}{(\AbLG\textsc{-w})}

\newcommand{\ICRL}{\mathcal{I}c\mathcal{RL}}
\newcommand{\IRL}{\vty{IRL}}
\newcommand{\RL}{\vty{RL}}
\newcommand{\CanRL}{\mathcal{C}an\mathcal{RL}}
\newcommand{\CanIRL}{\mathcal{C}an\mathcal{IRL}}
\newcommand{\GBL}{\mathcal{GBL}}
\newcommand{\CanGBL}{\mathcal{C}an\mathcal{GBL}}
\newcommand{\IGBL}{\mathcal{IGBL}}
\newcommand{\LG}{\vty{LG}} 
\newcommand{\LGm}{\mathcal{LG}^-}
\newcommand{\AbLG}{\vty{A}\mathit{b}\vty{LG}}
\newcommand{\CA}{\vty{CA}}

\newcommand{\SIRM}{\mathcal{S}i\mathcal{RM}}

\newcommand{\GRP}{\vty{G}\mathit{rp}}

\newcommand{\TRIV}{{\vty{T}\!riv}}

\HeadingsInfo{Gil\hyp{}F\'erez, Lauridsen, and Metcalfe}{Integrally~Closed~Residuated~Lattices}

\begin{document}

\setcounter{page}{1}     

 




\threeAuthorsTitle{J. Gil\hyp{}F\'erez}{F.\ts M.  Lauridsen}{G. Metcalfe}{Integrally Closed Residuated Lattices}


   
\PresentedReceived{Name of Editor}{Month day, year}


\begin{abstract}
A residuated lattice is said to be {\em integrally closed} if it satisfies the quasiequations\, $xy \leq x \implies y \leq \ut$\, and\, $yx \leq~x \implies y \leq \ut$, or equivalently, the equations\, $x \ld x \eq \ut$\, and\, $x \rd x \eq \ut$. Every integral, cancellative, or divisible residuated lattice is integrally closed, and, conversely, every bounded integrally closed residuated lattice is integral. It is proved that the mapping $a \mapsto (a \ld \ut)\ld \ut$ on any integrally closed residuated lattice is a homomorphism onto a lattice\hyp{}ordered group. A Glivenko\hyp{}style property is then established for varieties of integrally closed residuated lattices with respect to varieties of lattice\hyp{}ordered groups, showing in particular that integrally closed residuated lattices form the largest variety of residuated lattices admitting this property with respect to lattice\hyp{}ordered groups. The Glivenko property is used to obtain a sequent calculus admitting cut\hyp{}elimination for the variety of integrally closed residuated lattices and to  establish the decidability, indeed PSPACE-completenes, of its equational theory. Finally, these results are related to previous work on (pseudo) BCI\hyp{}algebras, semi\hyp{}integral residuated pomonoids, and Casari's comparative logic.
\end{abstract}




\section{Introduction} 

A {\em residuated lattice\hyp{}ordered monoid} ({\em residuated lattice} for short) is an algebraic structure $\m A=\pair{A,\mt,\jn,\pd,\ld,\rd,\ut}$ of type $\langle 2,2,2,2,2,0 \rangle$ such that $\langle A,\mt,\jn \rangle$ is a lattice, $\langle A,\pd,\ut \rangle$ is a monoid, and $\ld,\rd$ are left and right residuals of $\pd$ in the underlying lattice order, i.e., for all $a,b,c\in A$,
\[
b \le a \ld c \iff  ab \le c  \iff a \le c \rd b.
\]
These structures form a variety (equivalently, equational class) $\RL$ and provide algebraic semantics for substructural logics, as well as encompassing well\hyp{}studied classes of algebras such as lattice\hyp{}ordered groups ($\ell$-groups for short) and lattices of ideals of rings with product and division operators (see, e.g.,~\cite{JT02,BT03,GJKO07,MPT10}).

It is proved in~\cite{BCGJT03} that a residuated lattice $\m{A}$ is {\em cancellative} --- i.e., satisfies the monoid quasiequations $xy \eq xz \implies y\eq z$ and $yx \eq zx \implies y\eq z$ --- if and only if it satisfies the  equations $x \ld xy \eq y$ and $yx \rd x \eq y$. Cancellative residuated lattices hence form a variety $\CanRL$   that subsumes the varieties of $\ell$-groups $\LG$ and their negative cones $\LGm$, but excludes many important residuated lattices studied in logic (e.g., all non\hyp{}trivial Brouwerian algebras). In~\cite{MT10} it is proved that varieties of cancellative residuated lattices satisfying a further condition are categorically equivalent to varieties of lattice\hyp{}ordered groups with a co\hyp{}nucleus. Despite this rich structural theory, however, no analytic (cut\hyp{}free) proof system is known for $\CanRL$ and the decidability of its equational theory is still an open problem. The same issues arise also for varieties of cancellative residuated lattices satisfying one or both of the equations $xy \eq yx$ (commutativity) and  $x \le \ut$ (integrality).

In this paper we study residuated lattices that satisfy a weaker cancellation property and are closely related to other algebras for non\hyp{}classical logics, including (pseudo) BCI\hyp{}algebras~\cite{Ise66,KK92,KB06,DY08}, semi-integral residuated pomonoids~\cite{RvA00,EK18},  Dubreil\hyp{}Jacotin semigroups~\cite[Chap.~12--13]{Bly05}, and algebraic semantics for Casari's comparative logic~\cite{Cas87,Cas89,Cas97,Pao00a,Met06}. Following Fuchs~\cite[Chap.~XII.3]{Fuc63}, a residuated lattice $\m{A}$ is said to be {\em integrally closed} if it satisfies the (ordered monoid) quasiequations\, $xy \leq x \implies y \leq \ut$\, and\, $yx \leq x \implies y \leq \ut$, or equivalently, the equations $x \ld x \eq \ut$ and $x \rd x \eq \ut$. If $\m{A}$ is conditionally complete (i.e., every upper-bounded non-empty subset of $A$ has a least upper bound), then being integrally closed is equivalent to being integrally closed in the ordered group sense, namely, $a^n \le b$ for all $n \in \mathbb{N}{\setminus}\{0\}$ implies $a \le \ut$ (see~\cite[Chap.~XII.3]{Fuc63} for details).

Let us denote the variety of integrally closed residuated lattices by $\ICRL$. Clearly, every cancellative residuated lattice belongs to $\ICRL$. This is also the case for any {\em integral} residuated lattice; indeed, any bounded integrally closed residuated lattice is integral. Since any product of an integral residuated lattice and an $\ell$-group is integrally closed,  by~\cite[Cor.~5.3]{GT05}, this is the case in particular for all {\em GBL\hyp{}algebras}, residuated lattices satisfying the divisibility property: if $a \le b$, then there exist $c,d$ such that $a=cb=bd$. Like $\CanRL$, the variety $\GBL$ of GBL\hyp{}algebras has a rich structure theory (see, e.g.,~\cite{GT05,MT10}), but no analytic (cut\hyp{}free) proof system is known and the decidability of its equational theory is open.  Figure~\ref{figVarRL} depicts relationships between these and other varieties of integrally closed residuated lattices, using the prefixes $\mathcal{C}an$ and $\mathcal{I}$, respectively, to denote the cancellative and integral members of a variety, and $\TRIV$ to denote the trivial variety.

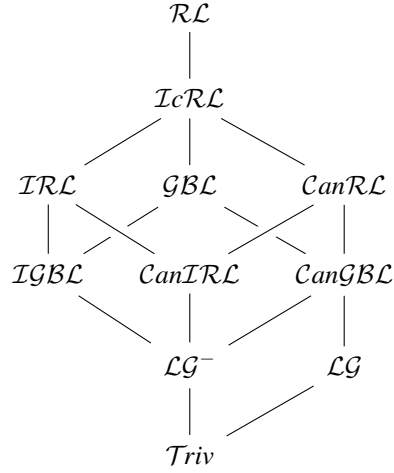
\begin{figure}[t] 
{\small
\begin{center}
\begin{tikzcd}[column sep = small, row sep = normal, arrows = dash]
{} & {\RL} & {} \\
{} & {\ICRL}\ar[u] & {} \\
{\IRL}\ar[ur] & {\GBL}\ar[u]\ar[rd] & {\CanRL}\ar[ul] \\
{\IGBL}\ar[u]\ar[ur] & {\CanIRL}\ar[ul, crossing over]\ar[ur, crossing over] & {\CanGBL}\ar[u] \\
{} & {\LGm}\ar[ul]\ar[u]\ar[ur] & {\LG}\ar[u] \\
{} & {\TRIV}\ar[u]\ar[ur] & {}
\end{tikzcd}
\end{center}}
\caption{Varieties of Integrally Closed Residuated Lattices}
\label{figVarRL}
\end{figure}

In Section~\ref{Sec:structure}, we prove that the mappings $a \mapsto a \ld \ut$ and $a \mapsto \ut \rd a$ coincide for any integrally closed residuated lattice $\m{A}$ and that $a \mapsto (a\ld \ut)\ld \ut$ defines a residuated lattice homomorphism from $\m{A}$ onto an $\ell$-group.  We use this result to establish a Glivenko\hyp{}style property (studied for varieties of pointed residuated lattices in~\cite{GO06}) for varieties of integrally closed residuated lattices with respect to varieties of $\ell$-groups (Theorem~\ref{thm:Glivenko:varieties:integrally-closed}), showing in particular that $\ICRL$ is the largest variety of residuated lattices admitting this property with respect to $\LG$ (Corollary~\ref{cor:largestGlivenkoLgroups}).

In Section~\ref{Sec:proof-theory-and-decidability}, we exploit the Glivenko property for $\ICRL$ to obtain a sequent calculus admitting cut\hyp{}elimination by extending the standard sequent calculus for $\RL$ (see, e.g.,~\cite{GJKO07,MPT10}) with a non\hyp{}standard weakening rule. As a consequence, we obtain the decidability,  indeed PSPACE\hyp{}completeness, of the equational theory of  $\ICRL$ (Theorem~\ref{thm:integrally-closed-decidable}). In Section~\ref{Sec:BCIetc}, these results are related to previous work on (pseudo) BCI\hyp{}algebras~\cite{Ise66,KK92,DY08} and semi\hyp{}integral residuated pomonoids~\cite{RvA00,EK18}. In particular, we prove that the equational theory of $\ICRL$ is a conservative extension of the equational theories of these classes (Theorem~\ref{thm:eq-MTRL-cons-ext-SIRM}), noting that the sequent calculus defined for BCI\hyp{}algebras corresponds  to a calculus used to prove decidability in~\cite{KK92}.  Finally, in Section~\ref{Sec:L-Pregroups}, we prove that the equational theory of a variety of algebras for  Casari's comparative logic~\cite{Cas87,Cas89,Cas97,Pao00a,Met06} is a conservative extension of the equational theory  of commutative integrally closed residuated lattices (Theorem~\ref{thm:lpregroupsconservative}). In this case, the sequent calculus corresponds to a system used in~\cite{Met06} to establish the decidability of Casari's comparative logic.


\section{The Structure of Integrally Closed Residuated Lattices} \label{Sec:structure}

In this section we establish some basic facts about the structure of integrally closed residuated lattices. We then use these facts to establish a Glivenko property for varieties of integrally closed residuated lattices with respect to varieties of $\ell$-groups. Let us recall first that every integral residuated lattice is integrally closed, and show that in the presence of a greatest or least element the converse is also true.

\begin{lemma}\label{lem:bounded-MnRL-are-integral}
Any upper or lower bounded integrally closed residuated lattice is integral.  
\end{lemma}
\begin{proof}
Suppose that $\top$ is the greatest element of an integrally closed residuated lattice $\m{A}$. Then $a \pd \top \leq \top$  and hence $a \leq \top \rd \top = \ut$ for all $a \in A$. So $\m{A}$ is integral. Moreover, any residuated lattice with a least element $\bot$ has a greatest element $\bot \rd\bot$, so must also be integral. 
\end{proof} 

\noindent
Since every finite residuated lattice is bounded, we obtain the following description of finite integrally closed residuated lattices.

\begin{corollary}\label{cor:finite-MnRLs-are-integral}
A finite residuated lattice is integrally closed if and only if it is integral.
\end{corollary}

\noindent
There are integrally closed residuated lattices that are not integral: for example, any non\hyp{}trivial $\ell$-group. The variety generated by all finite integrally closed residuated lattices is hence not $\ICRL$ --- that is, $\ICRL$ does not have the finite model property --- but rather the variety $\IRL$ of integral residuated lattices, which is known to be generated by its finite members~\cite{BvA05}.

A residuated lattice is called  {\em $\ut$\hyp{}cyclic} if the two unary operations $a \mapsto a \ld \ut$ and $a \mapsto \ut \rd a$ coincide. The next result shows that integrally closed residuated lattices have this property and that either one of the defining equations for this variety (relative to $\RL$) suffices to imply the other.

\begin{proposition}\label{prop:ucyclicmon} 
Any residuated lattice satisfying either $x \ld x \eq \ut$ or $x \rd x \eq \ut$ is integrally closed and  $\ut$\hyp{}cyclic. \end{proposition}
\begin{proof}
Let $\m{A}$ be a residuated lattice satisfying $x \ld x \eq \ut$, noting that the case where  $\m{A}$ satisfies $x \rd x \eq \ut$ is  symmetrical. Since any residuated lattice satisfies  $x \rd x \eq  (x \rd x) \ld  (x \rd x)$, we have that $\m{A}$ also satisfies $x \rd x \eq \ut$. That is, $\m{A}$ is integrally closed. Now consider any $a\in A$. By residuation, $a (a \ld \ut) \le~\ut$, so  $a (a \ld \ut)a \le a$, giving $(a \ld \ut)a \le a \ld a = \ut$ and hence $a \ld \ut \le \ut \rd a$. But, symmetrically, also  $\ut \rd a \le a \ld \ut$. So $\m{A}$ satisfies $x \ld \ut \eq \ut \rd x$ and is $\ut$\hyp{}cyclic.
\end{proof}

For any $\ut$\hyp{}cyclic residuated lattice $\m{A}$ and $a \in A$, we write $\altnot a$ to denote the common result $a \ld \ut = \ut \rd a$. The next lemma collects some useful properties of this operation for integrally closed residuated lattices.

\begin{lemma}\label{lem:alpha-is-hom}
Any one of the  following equations and quasiequations axiomatizes $\ICRL$ relative to the variety of $\ut$\hyp{}cyclic residuated lattices:
\begin{enumerate}[font=\upshape, label={(\roman*)}]
\item $\altnot(x \ld y) \eq \altnot y \rd \altnot x$
\item $\altnot(y \rd x) \eq \altnot x \ld \altnot y$
\item $x(\altnot x)y \leq \ut \implies y \leq \ut$
\item $y(\altnot x) x  \leq \ut \implies y \leq \ut$.
\end{enumerate}
\end{lemma}

\begin{proof}
For (i), let $\m{A}$ be any $\ut$\hyp{}cyclic residuated lattice and consider $a, b \in A$. Since $a(\altnot a)b \leq b$, it follows that $(\altnot a)b \leq a \ld b \leq \altnot\altnot (a \ld b)$ and $\altnot (a\ld b)(\altnot a)b \leq\ut$, yielding $\altnot (a\ld b) \leq  \altnot b \rd \altnot a$. Note also that  $a(a \ld b)(\altnot b \rd \altnot a)\altnot a \le b(\altnot b) \le \ut$. Hence if $\m{A}$ is integrally closed, it follows that $(a \ld b)(\altnot b \rd \altnot a) \le (\altnot a) \rd (\altnot a) = \ut$, yielding $\altnot b \rd \altnot a  \leq \altnot (a \ld b)$; that is, $\m{A}$ satisfies $\altnot(x \ld y) \eq \altnot y \rd \altnot x$. Conversely, if $\m{A}$ satisfies $\altnot(x \ld y) \eq \altnot y \rd \altnot x$, then $a \ld a \leq (\altnot a\rd \altnot a)(a\ld a) = \altnot (a\ld a)(a\ld a) \leq \ut$; that is,  $\m A$ satisfies $x\ld x \eq \ut$ and is integrally closed.  The proof for (ii) is symmetrical.

For (iii), consider first any integrally closed residuated lattice $\m{A}$  and $a,b\in A$. If $a(\altnot a)b \leq \ut$, then $(\altnot a) b\leq \altnot a$ and hence $b \leq \altnot a \ld \altnot a = \ut$; that is, $\m{A}$ satisfies $x(\altnot x)y \leq \ut \implies y \leq \ut$. Suppose next that $\m A$ is an $\ut$\hyp{}cyclic residuated lattice that satisfies $x(\altnot x)y \leq \ut \implies y \leq \ut$ and consider $a\in A$. Then $a(\altnot a)(\altnot a\ld\altnot a) \leq  \ut$ yields $\altnot a\ld \altnot a\leq \ut$. But also $\altnot a(a\rd a) a \leq \ut$,  yielding $a\rd a \leq \altnot a\ld\altnot a \leq \ut$. That is,  $\m A$ satisfies $x\rd x \eq \ut$ and is  integrally closed. The proof for (iv) is symmetrical. \end{proof}

For any $\ut$\hyp{}cyclic residuated lattice $\m{A}$, the map
\[
\alpha \colon A \to A; \quad a \mapsto \daltnot a
\]
is a nucleus on the induced partially ordered monoid $\pair{A, \leq, \pd, \ut}$, i.e., an increasing, order\hyp{}preserving,  idempotent map satisfying $\alpha(a)\alpha(b) \leq \alpha(ab)$ for all $a, b \in A$ (see, e.g.,~\cite[Lem.~3.35]{GJKO07}). Moreover, the image of $A$ under $\alpha$ can be equipped with the structure of a residuated lattice
\[
\m{A}_{\daltnot}  = \pair{\alpha[A],  \mt, \jn_{\daltnot}, \pd_{\daltnot}, \ld, \rd, \ut},
\]
where $a \jn_{\daltnot} b := \alpha(a \jn b)$ and $a \pd_{\daltnot} b := \alpha(a \pd b)$  (see, e.g., \cite[Thm.~3.34(4)]{GJKO07}).

Suppose now that $\m{A}$ is an integrally closed residuated lattice satisfying $\daltnot x \eq x$. Then for any $a \in A$,
\[
a(\altnot a)= \daltnot (a(\altnot a)) = \altnot (a \ld (\daltnot a)) = \altnot (a \ld a) = \altnot \ut = \ut.
\]
That is, $\m{A}$ satisfies $x(x \ld \ut) \eq \ut$ and is therefore an $\ell$-group (see~\cite[Sec.~2]{JT02} for the translations to the standard signature). In this case,  $\altnot$ is the group inverse operation and $\alpha$ is the identity map, so $\m{A}  = \m{A}_{\daltnot}$. On the other hand, if $\m{A}$ is an integral residuated lattice, then $\altnot a = \ut$ for all $a \in A$ and $\alpha$ maps every element to  $\ut$, so $\m{A}_{\daltnot}$ is trivial. More generally, if $\m{A}$ is integrally closed, then $\alpha$ and its image enjoy the following properties.

\begin{proposition}[]\label{prop:alpha}
Let $\m{A}$ be an integrally closed residuated lattice. 
\begin{enumerate}[font=\upshape, label={(\alph*)}]
\item	The map $\alpha \colon \m A \to \m A_{\daltnot}$ is a surjective homomorphism. 
\item	 $\m{A}_{\daltnot}$ is an $\ell$-group.
\end{enumerate}
\end{proposition}

\begin{proof}
(a)  Any nucleus on the induced partially ordered monoid of a residuated lattice preserves the monoidal structure and joins (see, e.g., \cite[Thm.~3.34(2)]{GJKO07}). By parts (i) and~(ii) of Lemma~\ref{lem:alpha-is-hom}, this nucleus also preserves the residual operations. It therefore suffices to show that $\alpha$ preserves binary meets. First note that since $(\altnot a)(\daltnot a) \leq \ut$, also $a(\altnot a)(\daltnot a) \leq a$, and, since $b(\altnot b) \leq \ut$, it follows that $a(\altnot a)b(\altnot b) (\daltnot a) \leq a$. Similarly,  $a(\altnot a)b(\altnot b)(\daltnot b) \leq b$, and hence
\[
a(\altnot a) b (\altnot b)(\daltnot a \mt \daltnot b) \leq a \mt b \leq \daltnot (a\mt b). 
\] 
By residuation, $a(\altnot a) b (\altnot b) (\daltnot a \mt \daltnot b)(\altnot (a\mt b)) \leq \ut$, and hence, applying part (iii) of Lemma~\ref{lem:alpha-is-hom} twice,  $(\daltnot a \mt \daltnot b)(\altnot (a\mt b)) \leq \ut$. By residuation again,
\[
\daltnot a \mt \daltnot b \leq \daltnot (a\mt b). 
\]
Since $\alpha$ is order\hyp{}preserving, $\daltnot (a \mt b) = \daltnot a \mt \daltnot b$ as required.

(b) That $\m{A}_{\daltnot}$ is an integrally closed residuated lattice follows immediately from part (a). But also for any $a \in A$, we have $\daltnot\alpha(a) = \alpha(\alpha(a)) = \alpha(a)$, so $\m{A}_{\daltnot}$ satisfies  $\daltnot x \eq x$ and is an $\ell$-group.
\end{proof}

\begin{proposition}\label{prop:torsion-free}
Every integrally closed residuated lattice is torsion\hyp{}free, i.e., satisfies the quasiequation\, $x^n \eq \ut \implies x \eq \ut$\,  for each $n \in \mathbb{N}{\setminus}\{0\}$.
\end{proposition}

\begin{proof}
Let $\m{A}$ be an integrally closed residuated lattice. We prove that $\m{A}$ satisfies\, $x^n \eq \ut \implies x \eq \ut$\, for each $n \in \mathbb{N}{\setminus}\{0\}$ by induction on $n$. The case $n = 1$ is trivial. For the inductive step, suppose that $n>1$ and $a^n = \ut$ for some $a \in A$. Then, since $\alpha \colon \m{A} \to \m{A}_{\daltnot}$ is a homomorphism,  $\alpha(a)^n = \alpha(a^n) = \alpha(\ut) = \ut$. But $\ell$-groups are torsion\hyp{}free, so $\alpha(a)=\ut$, yielding $\altnot a = \altnot\daltnot a = \altnot\alpha(a) = \altnot \ut = \ut$ and
\[
\ut = (\altnot a) a^n = (\altnot a) a a^{n-1} \leq a^{n-1} \leq  \altnot a = \ut.
\]
Hence $a^{n-1} = \ut$ and, by the induction hypothesis, $a = \ut$.     
\end{proof} 

We turn our attention now to varieties of integrally closed residuated lattices. Given any class $\vty{K} \subseteq \ICRL$, we denote by $\vty{K}_{\daltnot}$ the class $\{\m{A}_{\daltnot} \mid \m{A}\in\vty{K}\} \subseteq\LG$, recalling that $\LG$ denotes the variety of  $\ell$-groups. 

\begin{proposition}\label{prop:varieties:integrally-closed:lgroups}
Let $\vty{V}$ be any variety of integrally closed residuated lattices.
\begin{enumerate}[font=\upshape, label={(\alph*)}]
\item $\vty{V}_{\daltnot}$ is a variety of $\ell$-groups.
\item If $\vty{V}$ is defined relative to $\ICRL$ by a set of equations $E$, then $\vty{V}_{\daltnot}$ is defined relative to $\LG$ by $E$.
\end{enumerate}
Hence the map $\vty{V}\mapsto \vty{V}_{\daltnot}$ is an interior operator on the lattice of subvarieties of $\ICRL$ whose image is the lattice of subvarieties of $\LG$.
\end{proposition}

\begin{proof}
Let $\vty{V}$ be a variety of integrally closed residuated lattices defined relative to $\ICRL$ by a set of equations $E$ (e.g., the equational theory of $\vty{V}$), and let $\vty{W}$ be the variety of $\ell$-groups defined relative to $\LG$ by $E$. Clearly $\vty{W} = \vty W_{\daltnot} \subseteq \vty{V}_{\daltnot}$. But also each $\m{A}_{\daltnot} \in \vty{V}_{\daltnot}$ is, by Proposition~\ref{prop:alpha}, an $\ell$-group and a homomorphic image of $\m{A}\in\vty{V}$.  So  $\vty{V}_{\daltnot} \subseteq \vty{W}$. The last statement then follows from the observation that $\vty{V}_{\daltnot} = \vty{V}$ if and only if $\vty{V}\subseteq\LG$. 
\end{proof}

We now establish a correspondence between the validity of equations in an integrally closed residuated lattice $\m{A}$ and validity of equations in the $\ell$-group $\m{A}_{\daltnot}$, denoting the term algebra for residuated lattices over a fixed countably infinite set of variables by $\m{Tm}$.

\begin{lemma}\label{lem:A_daltnot}
For any integrally closed residuated lattice $\m{A}$ and $s,t \in \text{Tm}$,
\[
\m{A}_{\daltnot} \models s \leq t \iff \m{A} \models  \daltnot s \leq \daltnot t.  
\] 
\end{lemma}

\begin{proof}
Suppose first that $\m{A} \models  \daltnot s \leq \daltnot t$. Since $\m{A}_{\daltnot}$ is a homomorphic image of $\m{A}$, also $\m{A}_{\daltnot} \models  \daltnot s \leq \daltnot t$. But $\m{A}_{\daltnot}$ is an $\ell$-group, so $\m{A}_{\daltnot} \models  s \leq t$. 

Now suppose that $\m{A} \not\models  \daltnot s \leq \daltnot t$. Then there exists a homomorphism $\nu\colon\m{Tm}\to\m{A}$  such that $\nu(\daltnot s) \not\leq \nu(\daltnot t)$. Since $\alpha$ is a homomorphism from $\m{A}$ to $\m{A}_{\daltnot}$, we obtain a homomorphism $\alpha \circ \nu\colon\m{Tm}\to\m{A}_{\daltnot}$ such that
\[
\alpha \circ \nu(s) = \alpha(\nu(s)) = \daltnot \nu(s) = \nu(\daltnot s) \not\leq \nu(\daltnot t) = \alpha(\nu(t)) =\alpha \circ \nu(t).
\] 
Hence $\m{A}_{\daltnot} \not \models s \leq t$ as required. 
\end{proof}    

Following~\cite{GO06}, we will say that a variety $\vty{V}$ of residuated lattices  admits the \emph{(equational) Glivenko property} with respect to another variety  $\vty{W}$ of residuated lattices if for all $s,t \in\text{Tm}$,
\[
\vty{V} \models \ut \rd  (s \ld \ut) \leq  \ut \rd  (t\ld \ut)  \iff \vty{W} \models s \leq t \iff \vty{V} \models (\ut \rd s) \ld \ut \leq (\ut \rd t) \ld \ut,
\]
noting that if $\vty{V}$ is a variety of $\ut$\hyp{}cyclic residuated lattices, this simplifies to
\[
\vty{W} \models s \leq t \iff \vty{V} \models \daltnot s \leq \daltnot t.
\]
Let us also note the following useful consequence of this property.

\begin{proposition}\label{prop:Glivenko-for-inequations-s-leq-ut}
If $\vty V$ is a variety of residuated lattices admitting the Glivenko property with respect to a variety of residuated lattices $\vty W$, then for all $s\in \text{Tm}$,
\[
\vty W \models s\leq \ut \iff \vty V \models s\leq \ut.
\]
\end{proposition}

\begin{proof}
The equation $x \leq \ut\rd(x\ld \ut)$ and quasiequation $x\leq \ut \implies \ut\rd(x\ld \ut)\leq \ut$ are valid in all residuated lattices. Hence for all $s\in \text{Tm}$, 
\begin{align*}
\vty W \models s\leq \ut &\iff \vty V\models \ut\rd(s\ld\ut) \leq \ut\rd(\ut\ld\ut) \\
 &\iff \vty V\models \ut\rd(s\ld\ut) \leq \ut\\
  &\iff \vty V\models s \leq \ut. \tag*{\qedhere}
\end{align*}
\end{proof}

For integrally closed residuated lattices, we obtain the following pivotal result.

\begin{theorem}\label{thm:Glivenko:varieties:integrally-closed}
Any variety $\vty{V}$ of integrally closed residuated lattices admits the Glivenko property with respect to $\vty{V}_{\daltnot}$.
\end{theorem}

\begin{proof}
Suppose that $\vty V_{\daltnot} \models s\leq t$. For any $\m A\in \vty V$, it follows that $\m A_{\daltnot} \models s\leq t$, and hence $\m A\models \daltnot s\leq \daltnot t$, by Lemma~\ref{lem:A_daltnot}. So $\vty{V} \models  \daltnot s\leq \daltnot t$. The other implication follows from the fact that $\vty V_{\daltnot}\subseteq \vty V$ and $\vty V_{\daltnot} \models \daltnot x \eq x$.
\end{proof}

\begin{corollary}\label{cor:Glivenko:MTRL:vs:LG}
The variety of integrally closed residuated lattices admits the Glivenko property with respect to the variety of $\ell$-groups, and hence for all $s\in\text{Tm}$,
\[
\LG \models s \leq \ut \iff \ICRL \models s \leq \ut.
\] 
\end{corollary}

Theorem~\ref{thm:Glivenko:varieties:integrally-closed} shows that, in some sense, $\ICRL$ plays the same role for $\LG$ that the variety of Heyting algebras plays for the variety of Boolean algebras (see, e.g., \cite[Thm.~IX.5.3]{BD74}). Moreover, as the next result demonstrates, $\ICRL$ is the largest variety of residuated lattices that can play this role. 

\begin{theorem}
Let $\vty{V}$ be a variety of integrally closed residuated lattices that is axiomatized relative to $\ICRL$ by equations of the form $s \le \ut$. Then $\vty{V}$ is the largest variety of residuated lattices admitting the Glivenko property with respect to  $\vty{V}_{\daltnot}$. 
\end{theorem}
\begin{proof}
By Theorem~\ref{thm:Glivenko:varieties:integrally-closed}, $\vty{V}$ admits the Glivenko property with respect to $\vty{V}_{\daltnot}$. Now suppose that $\vty{W}$ is any variety of residuated lattices admitting the Glivenko property with respect to $\vty{V}_{\daltnot}$. By assumption, $\vty{V}$ is axiomatized relative  to $\ICRL$ by a set of equations $E$ of the form $s \le \ut$, so, by Proposition~\ref{prop:varieties:integrally-closed:lgroups}, the variety $\vty{V}_{\daltnot}$ is axiomatized relative to $\LG$ by $E$. But also by Proposition~\ref{prop:Glivenko-for-inequations-s-leq-ut}, all members of the variety $\vty{W}$ must satisfy all the equations in $E$ as well as $x \ld x \le \ut$. So $\vty{W} \subseteq \vty{V}$.  
\end{proof}

\begin{corollary}\label{cor:largestGlivenkoLgroups}
The variety of integrally closed residuated lattices is the largest variety of residuated lattices that admits the Glivenko property with respect to the variety of $\ell$-groups.
\end{corollary}


It is not  the case that every variety $\vty{V}$ of integrally closed residuated lattices is the largest variety of residuated lattices admitting the Glivenko property with respect to the corresponding variety  $\vty{V}_{\daltnot}$ of $\ell$-groups. For example,  the variety of commutative integrally closed residuated lattices corresponds to the variety of abelian $\ell$-groups. However, for any integral residuated lattice $\m A$, the $\ell$-group $\m A_{\daltnot}$ is trivial, so the largest variety admitting the Glivenko property with respect to the variety of abelian $\ell$-groups must contain all integral residuated lattices.

We conclude this section by describing a further syntactic relationship existing between  $\ICRL$ and the variety $\IRL$ of integral residuated lattices. Recall that for any residuated lattice ${\m A}$, the {\em negative cone} of ${\m A}$  is the residuated lattice $\m{A^-}$  with universe $A^{-} = \{a\in A \mid a \le \ut\}$, monoid and lattice operations inherited from ${\m A}$, and residuals $a \ld_{-} b:=(a \ld b)\mt\ut$ and $b \rd_{\!\!-}\, a :=(b \rd a)\mt\ut$, for $a, b \in A^{-}$. Define now inductively $\ut^- = \ut$, $x^- = x \mt \ut$ for each variable $x$,  $(s \ast t)^- =s^- \ast t^-$ for $\ast \in \{\mt,\jn,\pd\}$, 
\[
(s \ld t)^- = (s^- \ld t^-) \mt\ut,\quad\text{and}\quad
(s \rd t)^- = (s^- \rd t^-) \mt\ut.
\]
It is then straightforward (see~{\cite[Lem.~5.10]{JT02}}) to prove  that for any residuated lattice $\m{A}$ and  $s,t \in \text{Tm}$,
\[
\m{A^-} \models s \eq t \iff \m{A} \models s^- \eq t^-.
\]
Since the negative cone of an integrally closed residuated lattice is integral and an integral residuated lattice is integrally closed, we obtain the following result.

\begin{proposition}\label{prop:integraltranslation}
For any $s,t \in \text{Tm}$,
\[
\IRL \models  s \eq t \iff  \ICRL \models s^- \eq t^-.
\]
\end{proposition}


\section{Proof Theory and Decidability}\label{Sec:proof-theory-and-decidability}

In this section we obtain a sequent calculus for integrally closed residuated lattices as an extension of the standard sequent calculus for residuated lattices with a non\hyp{}standard weakening rule. We prove that this calculus admits cut\hyp{}elimination and obtain as a consequence a proof of the decidability, indeed PSPACE\hyp{}completeness, of the equational theory of integrally closed residuated lattices. 

A {\em (single\hyp{}conclusion) sequent} is an expression of the form $\Ga \seq t$ where $\Ga$ is a finite (possible empty) sequence of  terms $s_1,\ldots,s_n \in \text{Tm}$  and $t \in \text{Tm}$. Sequent rules, calculi, and derivations are defined in the usual way (see, e.g.,~\cite{GJKO07,MPT10}), and we say that a sequent $s_1,\ldots,s_n \seq t$ is {\em valid} in a class $\vty{K}$ of residuated lattices, denoted by $\models_{\vty{K}} s_1,\ldots,s_n \seq t$, if $\vty{K} \models s_1\cdots s_n \le t$, where the empty product is understood as $\ut$. 

As a base system we consider the sequent calculus $\pc{RL}$ presented in Figure~\ref{figRL}.  A sequent is derivable in $\pc{RL}$ if and only if it is valid  in $\RL$  (see, e.g.,~\cite{GJKO07,MPT10}), and $\pc{RL}$ admits cut\hyp{}elimination, i.e.,  there is an algorithm that transforms any derivation of a sequent in $\pc{RL}$ into a derivation of the sequent that does not use the cut rule.

\begin{figure}[t] 
\centering\small
\fbox{
\begin{minipage}{11.5 cm}
\[
\begin{array}{lcl}
\mbox{Identity Axioms} & & \mbox{Cut Rule}\\[.1in]
\infer[\pfa{\idr}]{s \seq s}{} & & \infer[\pfa{\cutr}]{\Ga_1, \Ga_2, \Ga_3 \seq u}{\Ga_2 \seq s & \Ga_1, s, \Ga_3 \seq u}\\[.2in]
\mbox{Left Operation Rules} & & \mbox{Right Operation Rules}\\[.1in]
\infer[\pfa{\tlr}]{\Ga_1, \ut, \Ga_2 \seq u}{\Ga_1, \Ga_2 \seq u} &  & \infer[\pfa{\trr}]{\seq\ut}{}\\[.15in]
\infer[\pfa{\rdlr}]{\Ga_1, t \rd s, \Ga_2, \Ga_3 \seq u}{\Ga_2 \seq s & \Ga_1, t, \Ga_3 \seq u} &  & 
\infer[\pfa{\rdrr}]{\Ga \seq t \rd s}{\Ga, s \seq t}\\[.15in]
\infer[\pfa{\ldlr}]{\Ga_1, \Ga_2, s \ld t, \Ga_3 \seq u}{\Ga_2 \seq s & \Ga_1, t, \Ga_3 \seq u} &  & 
\infer[\pfa{\ldrr}]{\Ga \seq s \ld t}{s, \Ga \seq t}\\[.15in]
\infer[\pfa{\pdlr}]{\Ga_1,  s \pd t, \Ga_2 \seq u}{\Ga_1,s,t, \Ga_2 \seq u} &  &
\infer[\pfa{\pdrr}]{\Ga_1, \Ga_2 \seq s \pd t}{\Ga_1 \seq s & \Ga_2 \seq t}\\[.15in]
\infer[\pfa{\alr_1}]{\Ga_1, s \mt t, \Ga_2 \seq u}{\Ga_1, s, \Ga_2 \seq u} & & 
\infer[\pfa{\orr_1}]{\Ga \seq s \jn t}{\Ga \seq s}\\[.15in]
\infer[\pfa{\alr_2}]{\Ga_1, s \mt t, \Ga_2 \seq u}{\Ga_1, t, \Ga_2 \seq u} & & 
\infer[\pfa{\orr_2}]{\Ga \seq s \jn t}{\Ga \seq t}\\[.15in]
\infer[\pfa{\olr}]{\Ga_1, s \jn t, \Ga_2 \seq u}{\Ga_1, s, \Ga_2 \seq u & \Ga_1, t, \Ga_2 \seq u}
 & \quad  \  & \infer[\pfa{\arr}]{\Ga \seq s \mt t}{\Ga \seq s & \Ga \seq t}
 \end{array}
\]
\caption{The Sequent Calculus $\pc{RL}$}
\label{figRL}
\end{minipage}}
\end{figure}

We define $\pc{IcRL}$ to be the sequent calculus consisting of the rules of $\pc{RL}$ together with the rule
\[
\infer[\pfa{\lgpw}]{\Ga, \De, \PI \seq u}{\Ga, \PI\seq u & \models_\LG \De \seq\ut}.
\]
This may be viewed as a special case of the weakening rule 
\[
\infer[\pfa{(\textsc{w})}]{\Ga, \De, \PI \seq u}{\Ga, \PI \seq u}
\]
where $\models_\LG \De \seq\ut$ is a decidable (indeed co\hyp{}NP\hyp{}complete) side\hyp{}condition~\cite{HM79,GM16}. The  side\hyp{}condition can also be understood proof\hyp{}theoretically as requiring a derivation of $\De \seq\ut$ in some calculus for $\ell$-groups, such as the analytic hypersequent calculus  provided in~\cite{GM16}.  Let us remark further that, since applications of $\lgpw$ can be pushed upwards in derivations, this rule can be replaced with  axioms of the form $\Ga, u, \Pi \seq u$ and  $\De \seq \ut$ with side\hyp{}conditions $\models_{\LG} \Ga \seq \ut$ and $\models_{\LG} \Pi\seq \ut$, and $\models_{\LG} \De \seq \ut$, respectively.

\begin{proposition}\label{prop:soundness-and-completeness-MnRL}
A sequent is derivable in  $\pc{IcRL}$ if and only if it is valid in all integrally closed residuated lattices.
\end{proposition}
\begin{proof}
For the right\hyp{}to\hyp{}left direction, we construct a Lindenbaum\hyp{}Tarski algebra. Namely, it can be shown (as usual) that the binary relation $\Theta$ defined on $\text{Tm}$ by
\[
u \mathrel{\Theta} v \enspace :\Longleftrightarrow \enspace u \seq v \text{ and } u \seq v \text{ are derivable in }\pc{IcRL}
\]
is a congruence on $\m{Tm}$ and  that the quotient $\m{Tm}/\Theta$ is an integrally closed (since $x \ld x \seq \ut$ and  $x \rd x \seq \ut$ are derivable in $\pc{IcRL}$) residuated lattice satisfying
\[
u/\Theta \le v/\Theta\ \iff\ u \seq v \text{ is derivable in }\pc{IcRL}.
\]
Suppose that $\models_{\ICRL} s_1,\ldots, s_n \seq t$ and consider the homomorphism from $\m{Tm}$ to $\m{Tm}/\Theta$ mapping each term $u$ to the equivalence class $u/\Theta$. Since $s_1 \cdots s_n \le t$ is valid in $\m{Tm}/\Theta$, it follows that $ s_1 \cdots s_n / \Theta \le t / \Theta$ and hence that $s_1 \cdots s_n \seq t$ is derivable in $\pc{IcRL}$. An application of $\cutr$ with the derivable sequent $s_1,\ldots,s_n \seq s_1\cdots s_n$ shows that also $s_1,\ldots,s_n \seq t$ is derivable in $\pc{IcRL}$. 

For the left\hyp{}to\hyp{}right direction, we recall (see, e.g.,~\cite{GJKO07,MPT10}) that the rules of $\pc{RL}$ preserve validity of sequents in $\RL$ and it suffices therefore to show that the rule $\lgpw$ preserves validity in $\ICRL$. Suppose then that $\models_{\ICRL} \Ga, \PI \seq u$ and $\models_{\LG} \De \seq \ut$. Writing $s_1$, $s_2$, and  $t$ for the products of the terms in $\Ga$, $\PI$, and $\De$, respectively, $\ICRL \models s_1 s_2  \leq u$ and $\LG \models t \leq \ut$. By Corollary~\ref{cor:Glivenko:MTRL:vs:LG}, we obtain $\ICRL \models t \leq \ut$ and hence $\ICRL \models s_1  t s_2 \leq u$. That is, $\models_{\ICRL}\Ga, \De, \PI \seq u$. 
\end{proof}

\begin{proposition}\label{prop:cut-elimination}
$\pc{IcRL}$ admits cut\hyp{}elimination.
\end{proposition}

\begin{proof}
It suffices (as usual) to prove that if there exist cut\hyp{}free derivations $d_1$ of $\Ga_2 \seq  s$ and $d_2$ of $\Ga_1,s,\Ga_3 \seq u$ in $\pc{IcRL}$, then there is a cut\hyp{}free derivation of $\Ga_1,\Ga_2,\Ga_3 \seq u$ in $\pc{IcRL}$, proceeding by induction on the lexicographically ordered pair $\langle c,h\rangle$ where $c$ is the term complexity of $s$ and $h$ is the sum of the heights of the derivations $d_1$ and $d_2$. The cases where the last steps in the derivations $d_1$ and $d_2$ are applications of rules of $\pc{RL}$ are standard (see, e.g., \cite[Chap.~4.1]{GJKO07}). We therefore just consider the cases where the last step is an application of the rule $\lgpw$.

Suppose first that $\Ga_2=\PI_1,\De,\PI_2$ and $d_1$ ends with
\[
\infer[\pfa{\lgpw}]{\PI_1,\De,\PI_2 \seq  s}{
 \infer{\PI_1,\PI_2 \seq  s}{\vdots\, d'_1} & \models_\LG \De \seq\ut}
\]
By the induction hypothesis, we obtain a cut\hyp{}free derivation $d_3$ of the sequent $\Ga_1,\PI_1,\PI_2,\Ga_3 \seq  u$  in $\pc{IcRL}$, and hence a cut\hyp{}free derivation  in $\pc{IcRL}$ ending with
\[
\infer[\pfa{\lgpw}]{\Ga_1,\PI_1,\De,\PI_2,\Ga_3 \seq  u}{
 \infer{\Ga_1,\PI_1,\PI_2,\Ga_3 \seq u}{\vdots\, d_3} & \models_\LG \De \seq\ut}
\]
Suppose next that $\Ga_3 =\PI_1,\De,\PI_2$ and $d_2$ ends with
\[
\infer[\pfa{\lgpw}]{\Ga_1,s,\PI_1,\De,\PI_2\seq  u}{
 \infer{\Ga_1,s,\PI_1,\PI_2 \seq  u}{\vdots\, d'_2} & \models_\LG \De \seq \ut}
\]
By the induction hypothesis, we obtain a cut\hyp{}free derivation $d_3$ of the sequent $\Ga_1,\Ga_2,\PI_1,\PI_2 \seq  u$  in $\pc{IcRL}$, and hence a cut\hyp{}free derivation  in $\pc{IcRL}$ ending with
\[
\infer[\pfa{\lgpw}]{\Ga_1,\Ga_2,\PI_1,\De,\PI_2 \seq  u}{
 \infer{\Ga_1,\Ga_2,\PI_1,\PI_2 \seq  u}{\vdots\, d_3} & \models_\LG \De \seq\ut}
\]
The analogous case where $\Ga_1 = \PI_1,\De,\PI_2$ is very similar.

Suppose finally that $\Ga_1,s,\Ga_3 = \PI_1,\De_1,s,\De_2,\PI_2$ and $d_2$ ends with
\[
\infer[\pfa{\lgpw}]{\PI_1,\De_1,s,\De_2,\PI_2\seq  u}{
 \infer{\PI_1,\PI_2 \seq  u}{\vdots\, d'_2} & \models_\LG \De_1,s,\De_2 \seq \ut}
\]
By Proposition~\ref{prop:soundness-and-completeness-MnRL}, we have $\models_{\ICRL} \Ga_2 \seq  s$ and hence $\models_{\LG}  \Ga_2 \seq  s$. But then also $\models_{\LG}  \De_1,\Ga_2,\De_2 \seq \ut$ and we obtain a cut\hyp{}free derivation  in $\pc{IcRL}$ ending with
\[
\infer[\pfa{\lgpw}]{\PI_1,\De_1,\Ga_2,\De_2,\PI_2\seq  u}{
  \infer{\PI_1,\PI_2 \seq  u}{\vdots\, d'_2}  & \models_\LG \De_1,\Ga_2,\De_2 \seq \ut}\qedhere
\] 
\end{proof} 

We use this cut\hyp{}elimination result to establish the decidability of the equational theory of $\ICRL$, noting that its quasiequational theory can be shown to be undecidable using the fact that the quasiequational theory of $\ell$-groups is undecidable~\cite{GG83}.

\begin{theorem}\label{thm:integrally-closed-decidable}
The equational theory of integrally closed residuated lattices is decidable, indeed PSPACE\hyp{}complete.
\end{theorem} 
\begin{proof}
For PSPACE\hyp{}hardness, it suffices to recall that the equational theory of integral residuated lattices is PSPACE\hyp{}complete~\cite{HT11} and consider the translation described in Proposition~\ref{prop:integraltranslation}. For inclusion, it suffices by Savitch's theorem, which states that NPSPACE = PSPACE~\cite{Sav70}, to observe that a non\hyp{}deterministic PSPACE algorithm for deciding validity of sequents is obtained by guessing and checking cut\hyp{}free derivations in $\pc{IcRL}$. The correctness of a derivation is checked branch by branch, recording only the branch of the derivation from the root to the current point. Note also that for the application of the rule $\lgpw$, we use the fact that the equational theory of $\LG$ is coNP\hyp{}complete~\cite{GM16} and therefore in PSPACE.
\end{proof}

Let us remark that the cut\hyp{}elimination argument of Proposition~\ref{prop:cut-elimination} applies also to sequent calculi for other varieties of integrally closed residuated lattices. First, let $\vty{V}$ be any variety of residuated lattices axiomatized relative to $\RL$ by a set of $\{\jn,\pd,\ut\}$\hyp{}equations. It is shown in~\cite[Sec.~3]{GJ13} that $\vty{V}$ can then be axiomatized by ``simple" equations of the form $s \le t_1 \jn \ldots \jn t_n$ where each of $s,t_1,\ldots,t_n$ is either $\ut$ or a product of variables and $s$ contains at most one occurrence of any variable. Moreover, a sequent calculus for $\vty{V}$ that admits cut\hyp{}elimination is obtained by adding to $\pc{RL}$ for each such equation $s \le t_1 \jn \ldots \jn t_n$, a ``simple" rule
\[
\infer{\Ga, \Psi(s), \Pi \seq  u}{\Ga, \Psi(t_1), \Pi \seq u & \dots & \Ga, \Psi(t_n), \Pi \seq u}
\]
where  $\Psi(\ut)$ is the empty sequence and  $\Psi(x_1 \cdots x_m)$ (for not necessarily distinct variables $x_1,\ldots,x_m$) is the sequence of corresponding metavariables $\Ga_{x_1},\dots,\Ga_{x_m}$. We obtain a sequent calculus for the variety $\vty{W}$ of integrally closed members of $\vty{V}$ that also admits cut\hyp{}elimination by adding the rule
\[
\infer[\pfa{(\vty{W}_{\daltnot}\textsc{-w})}]{\Ga, \De, \PI \seq u}{\Ga, \PI\seq u & \models_{\vty{W}_{\daltnot}} \De \seq\ut}
\]
In particular, a sequent calculus for the variety of commutative integrally closed residuated lattices is obtained by adding to $\pc{IcRL}$ the (left) exchange rule
\[
\infer[\pfa{\elr}]{\Ga_1, \Pi_1,\Pi_2, \Ga_2 \seq u}{\Ga_1,\Pi_2,\Pi_1,\Ga_2 \seq u}
\]
and replacing $\LG$ with the variety of abelian $\ell$-groups in the rule $\lgpw$. It can then also be shown, as in the proof of Theorem~\ref{thm:integrally-closed-decidable}, using the fact that abelian $\ell$-groups have a coNP\hyp{}complete equational theory, that the equational theory of commutative integrally closed residuated lattices is PSPACE\hyp{}complete. In general, however, decidability of such a variety of integrally closed residuated lattices $\vty{W}$ will depend not only on the decidability of the equational theory of $\vty{W}_{\daltnot}$, but also on the additional simple rules for $\vty{V}$.

Note finally that the equations $x \ld x \eq \ut$ and $x \rd x \eq \ut$ belong to the class $\mathcal{N}_2$ described in~\cite{CGT12}, but are not acyclic in the sense defined there and the method for constructing analytic sequent calculi in that paper therefore does not apply. Indeed, there can be no extension of $\pc{RL}$ with structural analytic rules (as defined in~\cite{CGT12}, and including the simple rules of~\cite{GJ13})  for $\ICRL$ that admits cut\hyp{}elimination. If this were the case, then, by~\cite[Thm.~6.3]{CGT12}, the variety $\ICRL$ would be closed under MacNeille completions. However, by Lemma~\ref{lem:bounded-MnRL-are-integral}, any bounded integrally closed residuated lattice is integral and the completion of an integrally closed residuated lattice $\m{A}$ will therefore be integrally closed only if $\m{A}$ is already integral. 


\section{Sirmonoids and Pseudo BCI\hyp{}Algebras}\label{Sec:BCIetc}

In this section we relate suitable reducts of integrally closed residuated lattices to semi\hyp{}integral residuated pomonoids,  studied in~\cite{RvA00,EK18}, and pseudo BCI\hyp{}algebras, defined in~\cite{DY08} as non\hyp{}commutative versions of BCI\hyp{}algebras~\cite{Ise66}.

A {\em residuated pomonoid} is a structure $\m{M} = \pair{M,\preceq,\pd,\ld,\rd,\ut}$ such that $\pair{M,\pd, \ut}$ is a monoid, $\preceq$ is a partial order on $M$, and $\ld,\rd$ are binary operations on $M$ satisfying $b \preceq a \ld c \iff  ab \preceq c  \iff a \preceq c \rd b$ for all $a,b,c\in M$. Such a structure is called {\em semi\hyp{}integral} if $\ut$ is a maximal element of $\pair{M,\preceq}$, or, equivalently, for all $a,b \in M$,
\[
a \preceq b \iff a \ld b = \ut \iff b \rd a = \ut.
\]
It is not hard to show that a semi\hyp{}integral residuated pomonoid (or \emph{sirmonoid} for short) may be identified with an algebraic structure $\m{S} = \pair{S,\pd, \ld,\rd,\ut}$ of type $\pair{2,2,2,0}$ satisfying the following equations and quasiequation:
\begin{enumerate}[font=\upshape, label={(\roman*)}]
\item $((x\ld z)\rd (y\ld z))\rd (x\ld y) \eq \ut$
\item $(y\rd x)\ld ((z\rd y)\ld (z\rd x)) \eq \ut$
\item $\ut\ld x \eq x$
\item $x\rd \ut \eq x$
\item $(x\pd y)\ld z \eq y\ld (x\ld z)$
\item $x\ld y\eq \ut \et y\ld x \eq \ut \implies x\eq y$.
\end{enumerate}
We let  $\SIRM$ denote the quasivariety of sirmonoids. 

Any group $\m{G} = \pair{G,\pd,^{-1},\ut}$ is term\hyp{}equivalent to a sirmonoid $\m{S}$ satisfying the equation $(x \ld \ut) \ld \ut \eq x$, noting that in this case, $a \preceq b$ if and only if $a = b$ for all $a, b \in S$. Given a group $\m{G}$, let  $a \ld b := a^{-1}\pd b$ and $b \rd a := b \pd a^{-1}$, and conversely, given a sirmonoid $\m{S}$ satisfying  $(x \ld \ut) \ld \ut \eq x$, let  $a^{-1} := a \ld \ut$. For convenience, we also call such a sirmonoid a {\em group} and denote the variety of these algebras by $\GRP$.

An algebraic structure $\m{B} = \pair{B,\ld,\rd,\ut}$ of type $\pair{2,2,0}$ satisfying the equations (i)--(iv) and quasiequation (vi) is called a {\em pseudo BCI\hyp{}algebra}. The $\{\ld,\rd,\ut\}$\hyp{}reduct of any sirmonoid is clearly a pseudo BCI\hyp{}algebra. More notably, every pseudo BCI\hyp{}algebra is a subreduct of a sirmonoid~\cite[Thm.~3.3]{EK18}, and the quasiequational theory of sirmonoids is therefore a conservative extension of the quasiequational theory of pseudo BCI\hyp{}algebras. In what follows, we consider to what extent similar relationships hold between sirmonoids and integrally closed residuated lattices.

\begin{lemma}\label{lem:reductsirmonoid}
A residuated lattice is  integrally closed if and only if its $\{\pd,\ld,\rd,\ut\}$\hyp{}reduct is a sirmonoid. 
\end{lemma}

\begin{proof}
Let $\m{A}$ be a residuated lattice. If its $\{\pd,\ld,\rd,\ut\}$\hyp{}reduct is a sirmonoid, then, since the induced partial order $\preceq$ is reflexive, $a \ld a = a \rd a = \ut$ for all $a \in A$, i.e., $\m{A}$ is integrally closed. Conversely, suppose that $\m{A}$ is integrally closed. Consider first $a,b \in A$ such that $a \ld b = \ut$. Then $\ut \le a\ld b$ and, by residuation twice, $\ut \le b \rd a$. Moreover, $\ut = \daltnot (a \ld b) = \daltnot a \ld \daltnot b$ and, since $\m{A}_{\daltnot}$ is an $\ell$-group, also $\ut = \daltnot b \rd \daltnot a = \daltnot (b \rd a) \ge b \rd a$. That is, $a \ld b = \ut$ implies $b \rd a = \ut$ and we define 
\[
a \preceq b \iff a \ld b = \ut \iff b \rd a = \ut.
\]
Since $\m{A}$ is integrally closed, $\preceq$ is reflexive. Also, if $a \preceq b$ and $b\preceq a$, then $\ut \le a \ld b$ and $\ut \le b \ld a$, yielding $a\le b$ and $b \le a$, i.e., $a=b$. So $\preceq$ is anti\hyp{}symmetric. Suppose now that $a \preceq b$ and $b\preceq c$. Then $\ut = a \ld b$ and $\ut = b \ld c$.  Hence $a\le b$ and $b \le c$, so $a \le c$, yielding $\ut \le a \ld c$. Note now that $x \ld z \le (x \ld y)\pd(y\ld z)$ holds in all $\ell$-groups and hence $(x \ld z) \pd (((x \ld y)\pd(y\ld z)) \ld \ut) \le \ut$ holds in all integrally closed residuated lattices by Corollary~\ref{cor:Glivenko:MTRL:vs:LG}. So $\ut\le a \ld c = (a \ld c) \pd \ut = (a \ld c) \pd (((a \ld b)\pd(b\ld c)) \ld \ut) \le \ut$, i.e., $a \ld c = \ut$ and $a\preceq c$. That is, $\preceq$ is transitive and hence a partial order.

Moreover, for all $a,b,c \in A$,
\begin{align*}
b \preceq a \ld c & \iff  (ab) \ld c = b \ld (a \ld c) = \ut\\
&  \iff ab \preceq c\\
& \iff  c \rd (ab) =  (c \rd b) \rd a =  \ut\\
&  \iff a \preceq c \rd b.  
\end{align*}
That is, the $\{\pd,\ld,\rd,\ut\}$\hyp{}reduct of $\m{A}$ is a sirmonoid.
\end{proof}

Not every sirmonoid is a subreduct of an integrally closed residuated lattice, however. By Proposition~\ref{prop:torsion-free}, $\{\pd,\ld,\rd,\ut\}$\hyp{}subreducts of  integrally closed residuated lattices  satisfy  $x^n \eq \ut \implies x \eq \ut$ for all $n \ge 1$, but there are sirmonoids (e.g., finite groups), that do not satisfy all of these quasiequations. On the other hand, it is known that any sirmonoid satisfying $x \preceq \ut$ is a subreduct of an integral (and hence integrally closed) residuated lattice \cite{Kur05}. 

The quasiequational theory of integrally closed residuated lattices is, as we have just seen, not a conservative extension of the quasiequational theory of sirmonoids. However, as we will show in Theorem~\ref{thm:eq-MTRL-cons-ext-SIRM}, such a conservative extension result does hold if we restrict to  equational theories.

Consider any  sirmonoid $\m{S}$. As before, we denote by $\altnot a$ the common result of $a \ld \ut$ and $\ut \rd a$ for $a \in S$, and obtain a nucleus $\alpha \colon S \to S; \ a \mapsto \daltnot a$ on $\pair{S, \preceq, \pd, \ut}$ and a residuated pomonoid  (see, e.g.,~\cite[Thm.~3.34(1)\hyp{}(3)]{GJKO07})
\[
\m{S}_{\daltnot}  = \pair{\alpha[S], \preceq,  \pd_{\daltnot}, \ld, \rd, \ut} \quad \text{where } a \pd_{\daltnot} b := \alpha(a \pd b).
\]
We also obtain the following analogue of Proposition~\ref{prop:alpha}.

\begin{proposition}\label{prop:alphasirmonoids}
Let $\m{S}$ be a sirmonoid. 
\begin{enumerate}[font=\upshape, label={(\alph*)}]
\item	The map $\alpha \colon \m S \to \m S_{\daltnot}$ is a surjective homomorphism. 
\item	 $\m{S}_{\daltnot}$ is a group.
\end{enumerate}
\end{proposition}

\begin{proof}
(a) Since $\alpha\colon S\to S$ is a nucleus on $\pair{S, \preceq, \pd, \ut}$ and $\alpha(\ut) = \ut$, it follows that $\alpha$ is a surjective monoid homomorphism between $\pair{S, \pd, \ut}$ and $\pair{\alpha[S], \pd_{\daltnot}, \ut}$. Now, given  $a,b\in S$, notice that $a(\altnot b)b \preceq a$ and therefore $\altnot b \preceq a\ld(a\rd b)$. That is, $(\altnot b)\ld (a\ld(a\rd b)) = \ut$. But also $\altnot b b (a\ld(a\rd b)) \preceq a\ld (a\rd b)$ and hence $b (a\ld(a\rd b)) \preceq (\altnot b)\ld (a\ld (a\rd b)) = \ut$, yielding $a\ld(a\rd b) \preceq b\ld\ut = \altnot b$. So $\m S$ satisfies $x\ld (x\rd y) \eq \altnot y$. Analogously, $\m S$ satisfies $(x\ld y)\rd y \eq \altnot x$ and hence for all $a,b \in S$,
\[
(\altnot b)\rd (\altnot a) = \big((a\ld b) \ld ((a\ld b)\rd b)\big) \rd (\altnot a) = \big((a\ld b) \ld (\altnot a)\big) \rd (\altnot a) = \altnot (a\ld b).
\]
That is, $\m S$ satisfies $\altnot (x\ld y) \eq (\altnot y)\rd (\altnot x)$ and, by a symmetric argument, also $\altnot (y\rd x) \eq (\altnot x)\ld (\altnot y)$. Hence
\[
\alpha(a\ld b) = \daltnot (a\ld b) = \altnot((\altnot b)\rd (\altnot a)) = (\daltnot a)\ld (\daltnot b) = \alpha(a)\ld \alpha(b).
\]
Analogously, $\alpha(b\rd a) = \alpha(b)\rd \alpha(a)$, so $\alpha$ is a sirmonoid homomorphism. 

(b) It follows from the fact that $\alpha$ preserves the residuals that $\m{S}_{\daltnot}$ is a sirmonoid. To prove that $\m{S}_{\daltnot}$  is a group, it suffices to show that it satisfies  $\daltnot x \eq x$. But $\alpha$ is idempotent and hence  $\daltnot \alpha(a) = \alpha(a)$ for every $a \in S$ as required.
\end{proof}

Recall that  $s \preceq t$  is valid in a group $\m{G}$ if and only if $s \eq t$ is valid in $\m{G}$. Moreover,  in every sirmonoid $\m{S}$ the map $a \mapsto \altnot a$ is both antitone, using residuation, and monotone, since $\m{S} \models \altnot (x\ld y) \eq \altnot y\rd \altnot x$. Hence, since  $\m{S} \models \altnot\daltnot x \eq \altnot x$, also $\altnot s \preceq \altnot t$ is valid in a sirmonoid $\m{S}$ if and only if  $\altnot s \eq \altnot t$ is valid in $\m{S}$. The proof of the following result now mirrors the proof of Lemma~\ref{lem:A_daltnot}.

\begin{lemma}
For any sirmonoid $\m{S}$ and residuated monoid terms $s,t$,
\[
\m{S}_{\daltnot} \models s \eq t \iff \m{S} \models  \daltnot s  \eq  \daltnot t.
\]
\end{lemma}

Given any class $\vty{K}$ of sirmonoids, we let $\vty{K}_{\daltnot}$ denote the corresponding class of groups $\{\m{S}_{\daltnot} \mid \m{S}\in\vty{K}\}$. The proof of the following Glivenko\hyp{}style result  proceeds very similarly to the proof of Proposition~\ref{prop:varieties:integrally-closed:lgroups}  and is therefore omitted.

\begin{proposition}\label{prop:varieties:sirmonoids:groups}
Let $\vty{Q}$ be any quasivariety of sirmonoids defined relative to $\SIRM$ by a set of equations $E$. Then $\vty{Q}_{\daltnot}$ is a variety of groups defined relative to $\GRP$ by $E$, and for any residuated monoid terms $s,t$,
\[
\vty{Q}_{\daltnot} \models s \eq t \iff \vty{Q} \models \daltnot s \eq  \daltnot t.
\] 
\end{proposition}

In particular, we obtain the following Glivenko\hyp{}style property for $\SIRM$  with respect to the variety of groups. 

\begin{corollary}\label{cor:Glivenko:SIRM:vs:GRP}
For any residuated monoid terms $s,t$,
\[
\GRP \models s \eq t \iff \SIRM \models \daltnot s  \eq  \daltnot t.
\] 
\end{corollary}

We use this result to prove that the equational theory of $\ICRL$ is a conservative extension of the equational theory of $\SIRM$.  We call a sequent $s_1,\ldots,s_n \seq t$  an {\em $m$\hyp{}sequent} if $s_1,\ldots,s_n, t$ are residuated monoid terms, and say that it is valid in a class $\vty{K}$ of sirmonoids, denoted $\models_{\vty{K}} s_1,\ldots,s_n \seq t$, if $\vty{K} \models s_1\cdots s_n \preceq t$,  recalling that the empty product is understood as $\ut$.  

\begin{proposition}
An $m$\hyp{}sequent is derivable in $\pc{IcRL}$ if and only if it is valid in all sirmonoids.
\end{proposition}
\begin{proof}
For the right\hyp{}to\hyp{}left direction, suppose that an $m$\hyp{}sequent $\Ga \seq t$ is valid in all sirmonoids. By Lemma~\ref{lem:reductsirmonoid}, it is also valid in all integrally closed residuated lattices, and hence, by Proposition~\ref{prop:soundness-and-completeness-MnRL}, derivable in $\pc{IcRL}$.

For the left\hyp{}to\hyp{}right direction, it suffices to show that all the rules of $\pc{IcRL}$ apart from $\cutr$ preserve validity in $\SIRM$. For the key case of $\lgpw$, suppose that $\models_{\SIRM} \Ga, \PI \seq u$ and $\models_{\LG} \De \seq \ut$. Letting $s_1$, $s_2$, and  $t$ denote the products of the terms in $\Ga$, $\PI$, and $\De$, respectively, we obtain  $\SIRM \models s_1 s_2 \le u$ and $\LG \models t \leq \ut$. We claim that  $\GRP \models t \eq \ut$. Otherwise, since the free group on countably infinitely many generators can be totally ordered (see, e.g.,~\cite[Thm.~3.4]{CR16}), we would have an $\ell$-group in which $\ut < t$,  contradicting $\LG \models t \leq \ut$.  Hence, by Corollary~\ref{cor:Glivenko:SIRM:vs:GRP}, we obtain $\SIRM \models t \preceq \ut$. So $\SIRM \models s_1  t s_2 \preceq u$; that is, $\models_{\SIRM} \Ga, \De, \PI \seq u$. 
\end{proof}

\begin{theorem}\label{thm:eq-MTRL-cons-ext-SIRM}
The equational theory of integrally closed residuated lattices is a conservative extension of the equational theories of  sirmonoids and pseudo BCI\hyp{}algebras.
\end{theorem}  

By the previous result, the sequent calculus consisting of the rules of $\pc{IcRL}$ restricted to $m$\hyp{}sequents and omitting the rules for $\mt$ and $\jn$ is sound and complete for the variety of sirmonoids and admits cut\hyp{}elimination. Similarly, if we further remove the rules for $\pd$, we obtain a sound and complete calculus for the variety of pseudo BCI\hyp{}algebras that admits cut\hyp{}elimination.

\begin{corollary}
The equational theories of sirmonoids and pseudo BCI\hyp{}algebras are decidable. 
\end{corollary}

Similar results hold for {\em BCI\hyp{}algebras}~\cite{Ise66}, axiomatized relative to pseudo BCI\hyp{}algebras by the equation $x\ld y \eq y\rd x$, and {\em sircomonoids}~\cite{RvA00}, axiomatized relative to sirmonoids by $x\ld y \eq y\rd x$ or $x \pd y \eq y \pd x$. In particular, the equational theory of commutative integrally closed residuated lattices  is a conservative extension of the equational theories of sircomonoids and BCI\hyp{}algebras. Let us remark also that the decidability of the equational theory of BCI\hyp{}algebras was proved using a sequent calculus with the restricted version of $\lgpw$ in~\cite{KK92}.


\section{Casari's Comparative Logic}\label{Sec:L-Pregroups}

\begin{figure}[t!] 
\centering\small
\fbox{
\begin{minipage}{11.5 cm}
\[
\begin{array}{lcl}
\mbox{Identity Axioms} & & \mbox{Cut Rule}\\[.1in]
\infer[\pfa{\idr}]{s \seq s}{} & & \infer[\pfa{\cutr}]{\Ga_1, \Ga_2, \Ga_3 \seq \De_1,\De_2}{\Ga_2 \seq s,\De_1 & \Ga_1, s, \Ga_3 \seq \De_2}\\[.2in]
\mbox{Structural Rules} & &\\[.1in]
\infer[\pfa{\elr}]{\Ga_1, \Pi_1, \Pi_2, \Ga_2 \seq \De}{\Ga_1,\Pi_2,\Pi_1,\Ga_2 \seq \De} & & 
\infer[\pfa{\err}]{\Ga \seq \De_1,\Sigma_1,\Sigma_2,\De_2}{\Ga \seq \De_1,\Sigma_2,\Sigma_1,\De_2}\\[.2in]
\mbox{Left Operation Rules} & & \mbox{Right Operation Rules}\\[.1in]
\infer[\pfa{\tlr}]{\Ga_1, \ut, \Ga_2 \seq \De}{\Ga_1, \Ga_2 \seq \De} &  & \infer[\pfa{\trr}]{\seq\ut}{}\\[.15in]
\infer[\pfa{\flr}]{\zr \seq}{} &  & \infer[\pfa{\frr}]{\Ga\seq\De_1,\zr,\De_2}{\Ga \seq \De_1,\De_2}\\[.15in]
\infer[\pfa{\ilr}]{\Ga_1, s \to t, \Ga_2, \Ga_3 \seq \De_1, \De_2}{\Ga_2 \seq s,\De_2 & \Ga_1, t, \Ga_3 \seq \De_1} &  & 
\infer[\pfa{\irr}]{\Ga \seq s \to t, \De}{\Ga, s \seq t, \De}\\[.15in]
\infer[\pfa{\pdlr}]{\Ga_1,  s \pd t, \Ga_2 \seq \De}{\Ga_1,s,t, \Ga_2 \seq \De} &  &
\infer[\pfa{\pdrr}]{\Ga_1, \Ga_2 \seq s \pd t, \De_1,\De_2}{\Ga_1 \seq s, \De_1 & \Ga_2 \seq t, \De_2}\\[.15in]
\infer[\pfa{\alr_1}]{\Ga_1, s \mt t, \Ga_2 \seq \De}{\Ga_1, s, \Ga_2 \seq \De} & & 
\infer[\pfa{\orr_1}]{\Ga \seq \De_1,s \jn t,\De_2}{\Ga \seq \De_1,s,\De_2}\\[.15in]
\infer[\pfa{\alr_2}]{\Ga_1, s \mt t, \Ga_2 \seq \De}{\Ga_1, t, \Ga_2 \seq \De} & & 
\infer[\pfa{\orr_2}]{\Ga \seq \De_1,s \jn t,\De_2}{\Ga \seq \De_1,t,\De_2}\\[.15in]
\infer[\pfa{\olr}]{\Ga_1, s \jn t, \Ga_2 \seq \De}{\Ga_1, s, \Ga_2 \seq u & \Ga_1, t, \Ga_2 \seq \De}
 &  & \infer[\pfa{\arr}]{\Ga \seq \De_1,s \mt t,\De_2}{\Ga \seq \De_1,s,\De_2 & \Ga \seq \De_1,t,\De_2}
 \end{array}
\]
\caption{The Sequent Calculus $\pc{InCPRL}$}
\label{fig:InCPRL}
\end{minipage}}
\end{figure}

The results of the previous sections extend with only minor modifications to the setting of {\em pointed residuated lattices} (also known as {\em FL\hyp{}algebras}), consisting of residuated lattices with an extra constant operation $\zr$. As before, we call such an algebra {\em integrally closed} if it satisfies $x \ld x \eq \ut$ and $x \rd x \eq \ut$. It is then straightforward to prove analogues of Lemma~\ref{lem:bounded-MnRL-are-integral} and Proposition~\ref{prop:ucyclicmon}, simply adding ``pointed'' before every occurrence of ``residuated lattice".

An $\ell$-group can be identified with an integrally closed pointed residuated lattice satisfying $(x \ld \ut) \ld \ut \eq x$ and $\zr \eq \ut$. However, to show that $\alpha \colon \m{A} \to \m{A}; \ a \mapsto \daltnot a$ on an integrally closed pointed residuated lattice $\m{A}$ defines a homomorphism onto an $\ell$-group $\m{A}_{\daltnot}  = \pair{\alpha[A],  \mt, \jn_{\daltnot}, \pd_{\daltnot}, \ld, \rd, \ut, \alpha(\zr)}$, we need also $\alpha(\zr) = \ut$. Assuming this  condition, we obtain analogues of Propositions~\ref{prop:alpha} and~\ref{prop:varieties:integrally-closed:lgroups}, and Theorem~\ref{thm:Glivenko:varieties:integrally-closed}  for integrally closed pointed residuated lattices satisfying $\zr \ld \ut \eq \ut$.\footnote{Note, however, that our definition of the Glivenko property for pointed residuated lattices now diverges from the definition of~\cite{GO06}, which considers the operations $a \mapsto \zr \rd (a \ld \zr)$ and $a \mapsto (\zr \rd a) \ld \zr$.}

Let us turn our attention now to a particular class of algebras  introduced by Casari in~\cite{Cas89} (see also~\cite{Cas87,Cas97,Pao00a,Met06}) to model comparative reasoning in natural language. For any commutative pointed residuated lattice $\m{A}$, we write $a \to b$ for the common result of $a \ld b$ and $b \rd a$; we also define $\lnot a := a \to \zr$ and $a + b := \lnot a \to b$ and say that $\m{A}$ is {\em involutive} if it satisfies $\lnot \lnot x \eq x$. We call an involutive commutative integrally closed pointed residuated lattice satisfying $\zr \to \ut \eq \ut$ (or equivalently $\zr\pd\zr \eq \zr$) a {\em Casari algebra} (called a {\em lattice\hyp{}ordered pregroup} in~\cite{Cas89}). We denote the variety of Casari algebras by $\CA$ and the variety of abelian $\ell$-groups (Casari algebras satisfying $\zr \eq \ut$) by $\AbLG$. The reasoning described above yields the following Glivenko\hyp{}style property for Casari algebras, first established in~\cite{Met06}.

\begin{proposition}[{\cite[Prop.~1]{Met06}}] For any pointed residuated lattice terms s,t,
\[
\AbLG \models s \le t \iff \CA \models \lnot\lnot s \le \lnot\lnot t. 
\]
\end{proposition}

A sequent calculus for Casari algebras was defined in~\cite{Met06}. We consider here {\em multiple\hyp{}conclusion sequents} defined as expressions of the form $\Ga \seq \De$ where $\Ga$ and $\De$ are finite (possibly empty) sequences of pointed residuated lattice terms. Generalizing our  definition for single\hyp{}conclusion sequents, we say that a multiple\hyp{}conclusion  sequent $s_1,\ldots,s_n \seq t_1,\ldots,t_m$ is {\em valid} in a class $\vty{K}$ of pointed residuated lattices, denoted by $\models_{\vty{K}} s_1,\ldots,s_n \seq t_1,\ldots,t_m$, if $\vty{K} \models s_1\cdots s_n \le t_1 + \cdots + t_m$, where the empty product is understood as $\ut$ and the empty sum as $\zr$.
 
The multiple\hyp{}conclusion sequent calculus $\pc{CA}$ consists of the calculus $\pc{InCPRL}$ for involutive commutative pointed residuated lattices defined in Figure~\ref{fig:InCPRL} extended with the rule
\[
\infer[\pfa{\ablgpw}]{\Ga_1, \Ga_2 \seq \De_1, \De_2}{\Ga_1 \seq \De_1 & \models_{\AbLG} \Ga_2 \seq \De_2}.
\]
The next proposition collects some results from~\cite{Met06}, noting that these can also be easily established using the methods of the previous sections.

\begin{proposition}[{\cite[Thms.~3,~4,~and~7]{Met06}}]\
\begin{enumerate}[font=\upshape, label={(\alph*)}]
\item	A multiple\hyp{}conclusion sequent is derivable in $\pc{CA}$ if and only if it is valid in $\CA$.
\item	$\pc{CA}$ admits cut\hyp{}elimination.
\item	The equational theory of Casari algebras is decidable.
\end{enumerate}
\end{proposition}

We are now able to establish the main result of this section.

\begin{theorem}\label{thm:lpregroupsconservative}
The equational theory of  Casari algebras is a conservative extension of the equational theories of commutative integrally closed residuated lattices, sircomonoids, and BCI\hyp{}algebras. 
\end{theorem}
\begin{proof}
The equational theory of commutative integrally closed residuated lattices is a conservative extension of the equational theories of sircomonoids and BCI\hyp{}algebras by Theorem~\ref{thm:eq-MTRL-cons-ext-SIRM}. Hence it suffices to show that the equational theory of  Casari algebras is a conservative extension of the equational theory of commutative integrally closed residuated lattices. 

Let $\pc{CIcRL}$ be the sequent calculus $\pc{CA}$ restricted to single\hyp{}conclusion sequents (i.e., sequents of the form $\Ga \seq t$ where $\zr$ does not occur in $\Ga$ or $t$). Then a single\hyp{}conclusion sequent is derivable in $\pc{CIcRL}$ if and only if it is valid in all commutative integrally closed residuated lattices. It therefore suffices to show that if a single\hyp{}conclusion sequent is derivable in $\pc{CA}$, it is derivable in $\pc{CIcRL}$. 

To this end, a simple induction on the height of a cut\hyp{}free derivation  shows that whenever a sequent $\Ga \seq \De$ not containing any occurrence of $\zr$ is derivable in $\pc{CA}$, the sequence $\De$ must be non\hyp{}empty. In particular, no sequent of the form $\Ga \seq$, where $\zr$ does not occur in $\Ga$, is derivable in $\pc{CA}$. But then a  straightforward induction on the height of a cut\hyp{}free derivation shows that any single\hyp{}conclusion sequent derivable in $\pc{CA}$, must also be derivable in $\pc{CIcRL}$.
\end{proof} 

\paragraph{Acknowledgements.} 
The research reported in this paper was supported by  Swiss National Science Foundation (SNF) grant 200021\textunderscore 184693.


\bibliographystyle{sl}


\AuthorAdressEmail{Jos\'{e} Gil\hyp{}F\'{e}rez}{Mathematisches Institut (MAI)\\
University of Bern\\
Sidlerstrasse 5\\
3012 Bern, Switzerland}{gilferez@gmail.com}

\AdditionalAuthorAddressEmail{Frederik M\"{o}llerstr\"{o}m Lauridsen}{Institute  for  Logic,  Language  and  Computation (ILLC) \\
 University  of  Amsterdam\\
P.O.  Box 94242\\
1090 GE Amsterdam, The Netherlands}{f.m.lauridsen@uva.nl}

\AdditionalAuthorAddressEmail{George Metcalfe}{Mathematisches Institut (MAI)\\
University of Bern\\
Sidlerstrasse 5\\
3012 Bern, Switzerland}{george.metcalfe@math.unibe.ch}

\end{document}